\documentclass[11pt,leqeq]{article}
\usepackage{amssymb,amsthm}
\usepackage{amsmath}
\usepackage{amscd}
\usepackage{xy}
\xyoption{all}
\usepackage{graphicx}
\newtheorem{thm}{Theorem}
\newtheorem{prop}[thm]{Proposition} 

\newtheorem{defn}[thm]{Definition}

\newcommand{\des}{\displaystyle}

\date{}

\begin{document}
\setlength{\baselineskip}{16pt}
\title{On the Gaussian $q$-Distribution}
\author{Rafael D\'\i az
and Eddy Pariguan}
\maketitle

\begin{abstract}

We present a study of the Gaussian $q$-measure introduced by
D\'iaz and Teruel from a probabilistic and from a combinatorial
viewpoint.  A main motivation for the introduction of the Gaussian $q$-measure
is that its moments are exactly the $q$-analogues of the double factorial numbers.  We show that the Gaussian $q$-measure interpolates
between the uniform measure on the interval $[-1,1]$ and the
Gaussian measure on the real line.
\end{abstract}

\section{Introduction}

{\noindent}The main goal of this work is to describe explicitly the Gaussian
$q$-measure and show that it fits into a diagram
\[ \xymatrix @C=.6in
 {\mbox{ Lebesgue on } [-1,1] &  \mbox{Gaussian $q$-measure on } [-\nu, \nu ] \ar[r]^{\ \ \ \ \ \ \ \ q \rightarrow 1 }
 \ar[l]_{q \rightarrow 0 \ \ \ \ \ \ }&  \mbox{ Gaussian on } \mathbb{R},}
 \]
where $\nu=\nu(q)=\frac{1}{\sqrt{1-q}}.$
That is we are going to construct a $q$-analogue for the Gaussian
measure and show that as $q$ moves from $0$ to $1$ the Gaussian
$q$-measure interpolates, in the appropriated sense, from the
uniform measure on the interval $[-1,1]$ to the normal measure on
the real line. If we think of the parameter $q$ as time, we
see that the Gaussian $q$-measure provides a transition from the
uniform distribution on the interval $[-1,1]$ to the
normal distribution centered at the origin, so it
describes a process of specialization at the  origin with a
simultaneous spread of probabilities towards infinity.\\

{\noindent}Let us make a couple of  remarks about terminology. We
shall use $q$-density, $q$-distribution, etc, to refer to the
$q$-analogues of the corresponding classical notions. The point to
keep in mind is that we always replace Lebesgue measure $dx$ by the
Jackson $q$-measure $d_qx.$ Unfortunately, to our knowledge, there
is not available axiomatic definition for the later object. So, to
that extent, our terminology should be taken heuristically. The
problem of justifying axiomatically the terminology used, although
of great value for understanding the foundations of our approach to
the Gaussian $q$-distribution, will not be further discussed in this
work. Next we remark that the object of study of this work -- the
Gaussian $q$-measure -- is not the same, despite the choice of name,
as the $q$-Gaussian measures that have been studied in the
literature. As far as we know there are two different distributions
that are called the $q$-Gaussian distribution. One of them was
introduced by Tsallis et all in \cite{pra, tsa10},  and has been developed in many
works, see the book \cite{GT} and the references therein. That
construction is motivated by the fact that the $q$-Gaussian
distribution is the maximum entropy distribution with prescribed
mean and dispersion for the so called Tsallis or extended entropy \cite{tsa1};
also the $q$-Gaussian distribution is an exact stable solution of
the nonlinear Fokker-Planck equation \cite{plas, tsa5}. Recently, a central
limit theorem involving the $q$-Gaussian measure has been proven by Umarov,
Tsallis and Steinberg \cite{U}. The other definition has been
studied by several researchers in various works such as \cite{bo,
br3,bryc, Lee}. This type of $q$-Gaussian measure is motivated by
the fact that it is the orthogonal measure associated with a certain
family of polynomials called the $q$-Hermite polynomials. A key
fact is that, in both cases, the $q$-Gaussian measure is a piecewise
absolutely continuous measure with respect to the Lebesgue measure;
in contrast the Gaussian $q$-measure studied in this work is
piecewise absolutely continuous with respect to the Jackson $q$-measure,
i.e. we are not just changing the density to be integrated, we are
simultaneously changing the very notion of integration. Our
generalization is motivated mainly by the fact it yields the right
moments, i.e. the  moments of the Gaussian $q$-measure are the
$q$-analogues of the Pochhammer $2$-symbol, as one may expect
\cite{DP, CTT}.

\section{Gaussian $q$-measure}

The construction of the $q$-analogue of the Gaussian
measure introduced in \cite{CTT} and further studied in \cite{ED2,
ED1} requires only a few basic notions from $q$-calculus \cite{AND,
RA, Ch, HY}. Fix a real number $0 \leq q < 1$.  The $q$-derivative
of a map $f:\mathbb{R} \longrightarrow \mathbb{R}$ at $x \in
\mathbb{R}\setminus \{0\}$ is given by
$$\des{\partial_{q}f(x)=\frac{f(qx)-f(x)}{(q-1)x}}.$$
Notice that for $q=0$, a case often ruled out in the
literature, one gets that:
$$\des{\partial_{0}f(x)=\frac{f(x)-f(0)}{x}}.$$

For an integer $n \geq 1$ we have that $\des{\partial_{q}x^{n} =
[n]_{q}x^{t-1}}$ where
$\des{[n]_{q}=\frac{q^{n}-1}{q-1}=1+q+...+q^{n-1}}$. Inductively one
can show that
$$\partial_{q}^nx^{n}=[n]_{q}[n-1]_{q}[n-2]_{q}...[2]_q=[n]_{q}!=(1-q)^{-n} \prod_{i=1}^n(1-q^i)
=\frac{(1-q)_q^{n}}{(1-q)^{n}}, $$ where we have made use of the
notation
$$(a + b)_q^n=\prod_{i=0}^{n-1}(a + q^i b).$$

A right inverse for the $q$-derivative  is obtained via
the Jackson integral or $q$-integral. For $a,b \in \mathbb{R},$ the
Jackson or $q$-integral of $f: \mathbb{R} \longrightarrow
\mathbb{R}$ on $[a,b]$ is given by
$$\int_{a}^{b}f(x)d_qx=(1-q)\sum_{n=0}^{\infty}q^n(bf(q^nb)-af(q^na)).$$

Notice that for good enough functions if one lets $q$ approach $1$
then the $q$-derivative approach the Newton derivative, and the
Jackson integral approach the Riemann integral. Note also that for
$q=0$ we get that:
$$\int_{a}^{b}f(x)d_0x=bf(b)-af(a).$$
  It is easy to show that
$q$-integration has the following properties.
\begin{prop}{\em For $a,b,c \in \mathbb{R}$ the following
identities hold:
$$\begin{array}{ll}
 1.\displaystyle \int_{0}^{b}f(x)d_qx =
(1-q)b\sum_{n=0}^{\infty}q^nf(q^nb). & 2.\displaystyle
\int_{a}^{b}f(x)d_qx =
-\int_{b}^{a}f(x)d_qx. \\
 3. \displaystyle \int_{ac}^{bc}f(x)d_qx =
c\int_{b}^{a}f(cx)d_qx. & 4.\displaystyle
\int_{-b}^{0}f(x)d_qx=\int_{0}^{b}f(-x)d_qx. \\
  5.\displaystyle \int_{a}^{c}f(x)d_qx= \int_{a}^{b}f(x)d_qx +
\int_{b}^{c}f(x)d_qx. & 6.\displaystyle \int_{-b}^{b}f(x)d_qx=
\int_{0}^{b}(f(x) + f(-x))d_qx.
\end{array}$$}
\end{prop}

The identities above show the similitude between the
Riemann and Jackson integrals. However the reader should be aware of
the sharp distinctions between them. Notice that the
$q$-integral of a function $f$ on an interval $[a,b]$ depends on the
values of $f$ on the interval $[0,b]$. Consider the $q$-measure of
the interval $[a,b]$; by definition it is given by
$$m_q[a,b]=\int_{a}^{b}1_{[a,b]}d_qx=(b-a)+qa-q^lb, $$
where $l$ is the smallest integer such that $q^l < \frac{b}{a}.$
Note that for $q=0$ we get that $$m_0[a,b]=b-a.$$ Therefore, for
intervals, $m_0$ agrees with the Lebesgue measure. One can check
that $m_q$ is additive, i.e. if $a<b<c<d$ then
$$m_q([a,b] \sqcup [c,d])=m_q[a,b] + m_q[c,d],$$ and
also that $m_q$ is well-behaved under re-scalings, i.e. for $c >0$ we
have that
$$m_q[ca,cb]=cm_q[a,b] .$$
However the measures $m_q$ for $0<q<1$ fail to be translation
invariant, indeed we have that:
$$m_q[a+c,b+c]=m_q[a,b] + c(q-q^l).$$

In order to find the $q$-analogue of the Gaussian measure
we should find $q$-analogues for the main characters appearing in
the Gaussian measure, namely:
$$\sqrt[]{2\pi}, \  \infty, \ e^{-\frac{x^2}{2}}, \ x^n, \ dx.$$
The Lebesgue measure $dx$ is replaced by the Jackson $q$-measure $d_{q}x$.
The monomial $x^{n}$ remains unchanged. The $q$-analogue of
$e^{-\frac{x^2}{2}}$ is constructed in several steps. The
$q$-analogue of the exponential function $e^x$ is
$$e_{q}^x = \sum_{n=0}^{\infty}\frac{x^n}{[n]_{q}!}= \sum_{n=0}^{\infty}\frac{(1-q)^n }{(1-q)_q^{n}}x^n.$$

The function $e_{q}^x$ is such that  $e_q^0=1$ and $\partial_q e_q^x = e_q^x$.  Notice that the $q$-exponential $e_{q}^x$ interpolates
between $\frac{1}{1-x}$ as $q$ approaches $0$, and $e^x$ as $q$
approaches $1$; thus the $q$-exponential $e_{q}^x$ provides a
transition from the hyperbolic to the exponential regime. This
procedure is illustrated in Figure
\ref{exp1 tres dim} which shows how  $e_{q}^x$ changes as
$q$ varies.

\begin{figure}[h!]
  \centering
    \includegraphics[width=.8\textwidth]{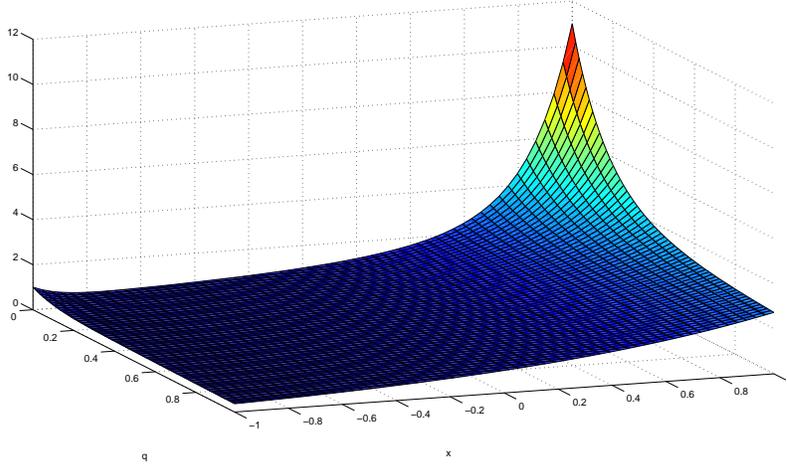}
    \caption{Plot of $e_{q}^x$ as function of $q$ and $x$.}
    \label{exp1 tres dim}
\end{figure}

The $q$-analogue of the identity $e^{x}e^{-x}=1$ is
$e_{q}^{x}E_{q}^{-x}=1$, where the function $E_{q}^{x}$ is given by
$$E_{q}^{x}= \sum_{n=0}^{\infty}q^{\frac{n(n-1)}{2}}\frac{x^n}{[n]_{q}!}=
\sum_{n=0}^{\infty}q^{\frac{n(n-1)}{2}}\frac{(1-q)^n
}{(1-q)_q^{n}}x^n.$$

The function $E_{q}^x$  is such that
$E_{q}^0=1$ and $\partial_qE_q^x=E_q^{qx}$. It is easy to see that
$E_{q}^x$ approaches $1+x$ as $q$ goes to $0$, and approaches $e^x$
as $q$ approaches to $1$; thus the $q$-exponential $E_{q}^x$
provides a transition from the linear to the exponential regime.
This interpolation is shown in Figure \ref{exp2 tres dim} which
shows how the graph of $E_{q}^x$ changes as $q$ varies.

\begin{figure}[h!]
  \centering
    \includegraphics[width=.8\textwidth]{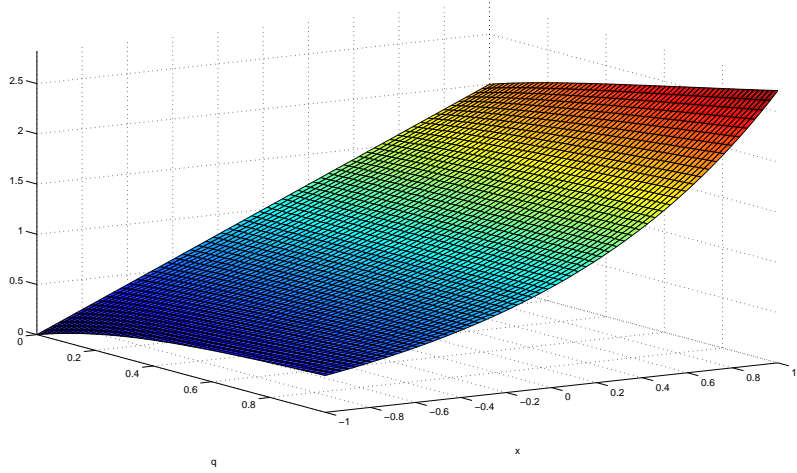}
    \caption{Plot of $E_{q}^x$ as a function of $q$ and $x$ .}
    \label{exp2 tres dim}
\end{figure}
\newpage
Finding the right $q$-analogue for $e^{-\frac{x^2}{2}}$
is a bit tricky. With hindsight we know that it is given by:
$$ E_{q^{2}}^{-\frac{q^{2}x^2}{[2]_q}}=
\sum_{n=0}^{\infty}\frac{(-1)^{n}q^{n(n+1)} }{(1+q)^{n}[n]_{q^{2}!}}x^{2n}=
\sum_{n=0}^{\infty}\frac{q^{n(n+1)} (q-1)^n  }{  (1-q^2)_{q^2}^{n} }x^{2n}.$$

Perhaps the most delicate issue is finding the
$q$-analogues for the integration limits. Remarkably the
$q$-analogue of an improper integral is a proper integral with
limits $-\nu$ and $\nu$ where
$$\nu=\nu(q)=\frac{1}{\sqrt{1-q}}.$$

Notice that $\nu$ approaches $1$ as
$q$ goes to $0$ and approaches $\infty$ as $q$ goes to $1$. The
normalization factor is also delicate. It turns out that the
$q$-analogue $c(q)$ of $\sqrt{2\pi}$ is given by

$$c(q) = \int_{-\nu}^{\nu} E_{q^2}^{\frac{-q^2x^2}{[2]_q}}d_qx =
2\int_{0}^{\nu} E_{q^{2}}^{\frac{-q^2x^2}{[2]_q}}d_qx =
2(1-q)\nu\sum_{n=0}^{\infty}q^n
E_{q^{2}}^{\frac{-q^2(q^n\nu)^2}{[2]_q}}, $$ or equivalently
$$c(q) = 2(1-q)^{\frac{1}{2}}\sum_{m=0}^{\infty}
\frac{(-1)^{m}q^{m(m+1)}}{(1-q^{2m+1})(1-q^2)_{q^2}^{m}}.$$
Note that $c(0)=2$ and that $c(q)$ approaches $\sqrt{2\pi}$ as $q$
goes to $1$; one may think of $\des \frac{c(q)^2}{2}$ as a being a
$q$-analogue for $\pi$, indeed one gets the following remarkably
identity
$$\pi = 2  \lim_{q \rightarrow 1}
\left( \sum_{m=0}^{\infty}
\frac{(-1)^{m}(1-q)^{\frac{1}{2}}q^{m(m+1)}}{(1-q^{2m+1})(1-q^2)_{q^2}^{m}}\right)^2.$$
The graph of $c(q)$ as a function of $q$ is shown in Figure
\ref{fig:cq}.

\begin{figure}[h!]
  \centering
    \includegraphics[width=.5\textwidth]{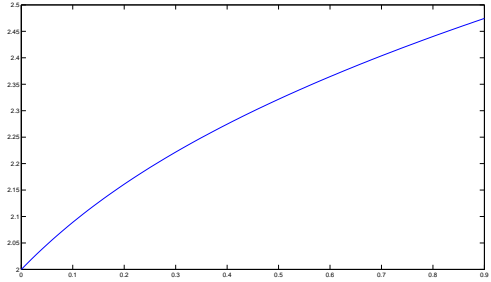}
    \caption{Plot of $c(q)$ as a function of  $q$. }
    \label{fig:cq}
\end{figure}

We are ready to introduce the Gaussian $q$-density.

\begin{defn}{\em
The Gaussian $q$-density is the functions $s_q:\mathbb{R}
\longrightarrow \mathbb{R}$ is given by
\[s_q(x)=
\left\{
\begin{array}{lc}
  0 & \mbox{for}\ x < -\nu, \\
  & \\
  {\displaystyle \frac{1}{c(q)}E_{q^{2}}^{\frac{-q^2x^2}{[2]_q}}} &
  \mbox{\ for\ }-\nu
\leq x \leq \nu \\
& \\
  0 &  \mbox{for}\ x > \nu
\end{array}
\right.
\]}
\end{defn}

\begin{figure}[h!]
  \centering
    \includegraphics[width=.8\textwidth]{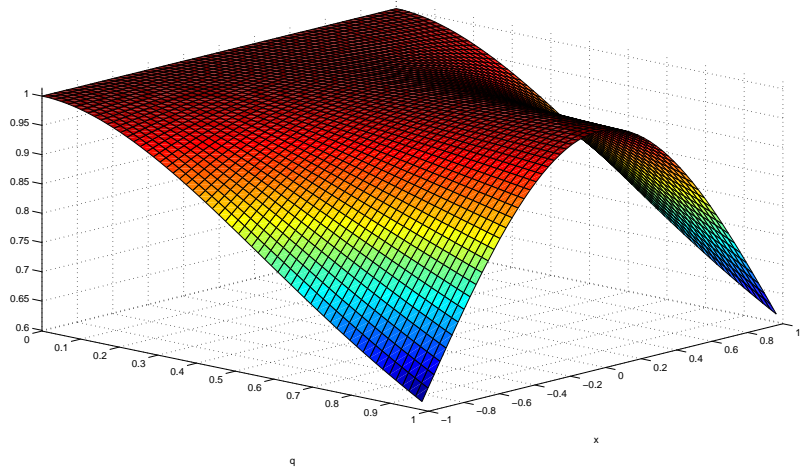}
    \caption{Plot of the Gaussian $q$-density as a function of $q$ and $x$.}
    \label{qcampana3d}
\end{figure}

\begin{thm}{\em The Gaussian $q$-density interpolates between the
uniform density on the interval $[-1,1]$ and the Gaussian density
on the real line.}
\end{thm}

\begin{proof}
We must show that $s_q$ converges  to  $\frac{1}{2}1_{[-1,1]}$ as
$q$ goes to $0$, and that $s_q$ converges to
$\frac{1}{\sqrt[]{2\pi}}e^{-\frac{x^2}{2}}$ as $q$ approaches $1$.
Both results are immediate from our previous remarks.
\end{proof}

The transition of the Gaussian $q$-density from the
uniform density on $[-1,1]$ to the Gaussian density on the real line
is shown in Figure \ref{qcampana3d}.

\section{Gaussian $q$-measure and $q$-combinatorics}

The reader may be wondering about the motivation behind our
definition of the Gaussian $q$-density $s_q$. It has been
constructed so that it generalizes the fact that the Gaussian
measure provides a bridge between measure theory and
combinatorics; indeed the moments of the Gaussian measure are
given by
$$\frac{1}{\sqrt[]{2\pi}}
\int_{-\infty}^{\infty}x^{n} e^{-\frac{x^2}{2}}dx=|M[n]|,$$
where $|M[n]|$ is the cardinality of the set  $M[n]$ of matchings
on $[n]=\{1,2,...,n\},$ i.e. the number of partitions of $[n]$ in
blocks of cardinality $2$. Thus the Gaussian measure has a clear
combinatorial meaning, this fact explains the role of graphs
in the computation of Feynman integrals
\cite{RDEP}.\\

Just as the basic object of study in combinatorics
is the cardinality of finite sets, the basic object of study in
$q$-combinatorics is the cardinality of $q$-weighted sets, i.e.
pairs $(x,\omega)$ where $x$ is a finite set and $\omega:x
\longrightarrow \mathbb{N}[q]$ is an arbitrary map. The
cardinality of such a pair is given by
$$|x, \omega|=\sum_{i\in x}\omega(i).$$

Let us now describe  \cite{ED1} the interpretation in terms of
$q$-combinatorics of the Gaussian $q$-measure. A matching $m$ on
$[n]$ is a sequence $m=\{(a_1,b_1),(a_2,b_2),..., (a_n,b_n)\}$
such that $a_i<b_i$, $a_1< a_2 <
\dots < a_n$, and ${[n]=\bigsqcup\{a_i,b_i\}}.$  Next we define a $q$-weight on $M[n]$ the set of  matchings on $[n]$. For
a pair $(a_i,b_i)$ in a matching  $m$ we set $((a_i,b_i))=\{j\in
[[2n]]: a_i<j<b_i\}.$ Also for an integer $i$ we set
$B_i(m)=\{b_j: 1\leq j < i \}$. The weight $\omega(m)$ of a
matching $m$ is defined as follows:
$$\displaystyle{\omega(m)=\prod_{(a_i,b_i) \in m}q^{|((a_i,b_i)) \setminus
B_i(m)|}=q^{\sum_{(a_i,b_i) \in m}|((a_i,b_i)) \setminus
B_i(m)|}}.$$

\begin{thm}\label{dens}{\em For $n\geq 0$ we have that
$$\frac{1}{c(q)}\int_{-\nu}^{\nu}x^{n}E_{q^{2}}^{\frac{-q^2x^2}{[2]_q}}d_{q}x=
|M[n], \omega |.$$}
\end{thm}

Since there are no matchings for a set of odd
cardinality we have that  $$|M[2n+1], \omega|=0.$$

One can show by induction that
$$|M[2n], \omega|=[2n-1]_q!!=[2n-1]_q[2n-3]_q.....[3]_q.$$
Therefore we have that
\begin{itemize}
\item $\frac{1}{c(q)}\des \int_{-\nu}^{\nu}x^{2n+1}E_{q^{2}}^{\frac{-q^2x^2}{[2]_q}}d_{q}x=0.$

\item $\frac{1}{c(q)}\des \int_{-\nu}^{\nu}x^{2n}E_{q^{2}}^{\frac{-q^2x^2}{[2]_q}}d_{q}x=
[2n-1]_q!!.$
\end{itemize}

The $q$-combinatorial interpretation of the Gaussian $q$-measure
is the starting point for our construction of $q$-measures of the
Jackson-Feynman type in \cite{ED2, ED1}. It would be interesting
to study the categorical analogues of these $q$-measures along the
lines of \cite{Blan2, RDEP}. The reader
should note that the formula above provides a $q$-integral
representation the $q$-analogue of
the Pochhammer $k$-symbol with $k=2$. A $q$-integral
representation for the general Pochhammer $q,k$-symbol is treated
in
\cite{CTT}. The integral representation of the Pochhammer
$k$-symbol is studied in
\cite{DP}.
\section{Gaussian $q$-distribution}

Let us study how probabilities are distributed on the real line
according to the Gaussian $q$-distribution.

\begin{prop}{\em For $0 \leq a < b \leq \nu$ we have

$$\frac{1}{c(q)}\int_{a}^{b}E_{q^{2}}^{\frac{-q^2t^2}{[2]_q}}d_qt=
\frac{1-q}{c(q)} \sum_{n=0}^{\infty}\frac{q^{n(n+1)} (q-1)^n
}{(1-q^{2n+1}) (1-q^2)_{q^2}^{n}}(b^{2n+1} - a^{2n+1} ).$$}
\end{prop}

\begin{proof}
\begin{eqnarray*}
\int_{a}^{b}E_{q^{2}}^{\frac{-q^2t^2}{[2]_q}}d_qt &=&
(1-q)\sum_{m=0}^{\infty}q^m(bE_{q^{2}}^{\frac{-q^2(q^mb)^2}{[2]_q}}-
aE_{q^{2}}^{\frac{-q^2(q^ma)^2}{[2]_q}})\\
&=&(1-q)\sum_{m,n=0}^{\infty}\frac{q^{n(n+1)} (q-1)^n q^{(2n+1)m}
}{ (1-q^2)_{q^2}^{n}}(b^{2n+1} - a^{2n+1} )\\
&=& (1-q)\sum_{n=0}^{\infty}\frac{q^{n(n+1)} (q-1)^n
}{(1-q^{2n+1}) (1-q^2)_{q^2}^{n}}(b^{2n+1} - a^{2n+1} ).
\end{eqnarray*}
\end{proof}

The reader may wonder about the convergence of the
series on the right hand side of the formula from the statement of
the previous theorem. Indeed the factors $(a^{2n+1} - b^{2n+1} )$
may suggest divergency, note however that the factors $q^{n(n+1)}$
ensure convergency.

\begin{defn}{\em
The Gaussian $q$-distribution $G_q:\mathbb{R}
\longrightarrow \mathbb{R}$ is given by
\[  G_q(x) =
\left\{
\begin{array}{lc}
  0 & \mbox{for}\ x < -\nu, \\
  & \\
  {\displaystyle \frac{1}{c(q)}\int_{-\nu}^{x}E_{q^{2}}^{\frac{-q^2t^2}{[2]_q}}d_qt} &
  \mbox{\ for\ }-\nu
\leq x \leq \nu \\
& \\
  1 &  \mbox{for}\ x > \nu
\end{array}
\right.
\]
}
\end{defn}

Below we need the following notation, if $A \subseteq
\mathbb{R}$ then as usual we define the characteristic function
$1_A$ as follows:
\[
1_A(x)=\left\{
\begin{array}{lc}
  0 & \mbox{for \ } x \mbox{\ not in  } \ A , \\

& \\
  1 & \mbox{for \ } x \mbox{\ in  } \ A.
\end{array}
\right.
\]

Next result provides explicit formulae for the Gaussian
$q$-distribution.

\begin{thm} {\em
For $x \in \mathbb{R}$ we have that:
$$G_q(x)=1_{[-(1-q)^{-1/2}, (1-q)^{-1/2} ]}(x)\left( \frac{1}{2} +
\frac{1-q}{c(q)} \sum_{n=0}^{\infty}\frac{q^{n(n+1)}
(q-1)^n  }{(1-q^{2n+1}) (1-q^2)_{q^2}^{n}}x^{2n+1}\right) +
1_{((1-q)^{-1/2}, \infty)}(x).$$ }
\end{thm}
\begin{proof}
The result follow from the previous proposition, the fact that
$s_q(x)$ is symmetric about the origin, and the straightforward
set theoretical identities: $$[-\nu,x]\sqcup [x,0]=[-\nu,0]
\mbox{ \ \  for \ \ } -\nu
\leq x \leq0;$$ $$[-\nu,x]= [-\nu, 0]\sqcup [0,x]\mbox{ \ \ for \ \  } 0 \leq x
\leq \nu.$$
\end{proof}

\begin{thm}{\em The Gaussian $q$-distribution interpolates between the
uniform distribution on the interval $[-1,1]$ and the Gaussian
distribution on the real line.}
\end{thm}

\begin{proof}
We must show that $G_q(x)$ converges to
$$\frac{1+x}{2}1_{[-1,1]}$$ as $q$ goes to $0$; and  converges  to
$$\frac{1}{\sqrt[]{2\pi}}\int_{-\infty}^{x} e^{-\frac{t^2}{2}}dt$$ as
$q$ approaches $1$. Both results are immediate from our previous
considerations.
\end{proof}

{\noindent} The transition of the Gaussian $q$-Distribution from
$\frac{1+x}{2}1_{[-1,1]}$ to the Gaussian distribution as $q$ moves
from $0$ to $1$ is shown in Figure \ref{qcampana3d2}.

\begin{figure}[h!]
  \centering
    \includegraphics[width=.8\textwidth]{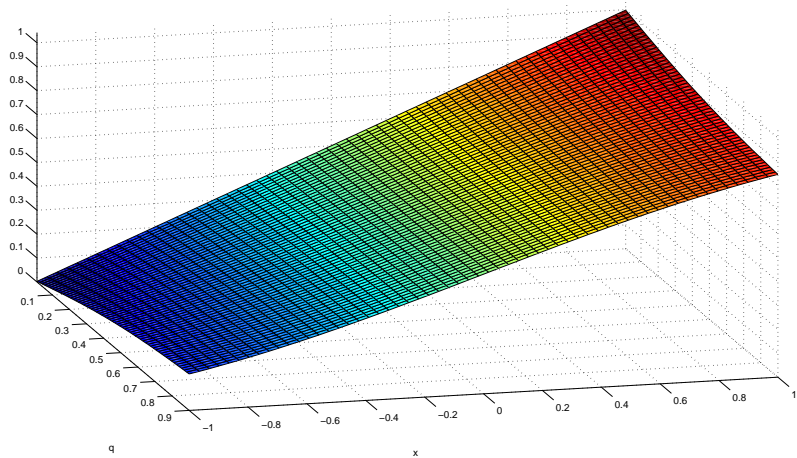}
    \caption{Plot of the Gaussian $q$-distribution for $-1 \leq x \leq 1$.}
    \label{qcampana3d2}
\end{figure}

\section{Conclusion}

The countless applications of the Gaussian measure in mathematics, science and engineering,
suggest that the Gaussian $q$-measure may also find its share of
applications. We showed that as $q$ moves from $0$ to $1$ the
Gaussian $q$-density and the Gaussian $q$-distribution interpolate
between the uniform density and the uniform distribution on the
interval $[-1,1]$ to the Gaussian density and the Gaussian
distribution. Note that the transition from specialization to
uniformity is a common phenomena both in nature and in
mathematics. Indeed, we are used the see objects breaking apart but
we seldom see them coming back together to form a unity from the
many pieces. Likewise in mathematics the transfer of heat in a
compact manifold will eventually end up with a uniform temperature
trough out the manifold, regardless of the fact that the initial
distribution of heat many have been localized around some point.
The reverse transition form uniformity to specialization occurs
less often, yet it is  a standard phenomena in certain domains of
nature, phenomena of such type  play a most fundamental role in
some chemical interactions and in microbiology. For that reason we
believe that our Gaussian $q$-measure may find some applications
in those fields of study.

\subsection*{Acknowledgment}

Our thanks to Nicol\'as Vega and Camilo Ortiz for assisting us with
the figures which were drawn using the computer software MATLAB.

\noindent ragadiaz@gmail.com\\
\noindent Facultad de Administraci\'on,
Universidad del Rosario, Bogot\'a, Colombia \\

\noindent epariguan@javeriana.edu.co \\
Departamento de Matem\'aticas, Pontificia Universidad Javeriana,
Bogot\'a, Colombia.

\end{document}